\documentclass{IEEEtran}
\usepackage[utf8]{inputenc}
\usepackage{braket}
\usepackage{amsmath}
\usepackage{booktabs}
\usepackage{amsthm}
\usepackage{array}
\usepackage{tikz}
\usepackage{hyperref} 
\usepackage{amsfonts}
\usepackage{bm}
\usepackage[justification=centering]{caption}
\usepackage{mathdots}
\usepackage{adjustbox}
\usepackage{comment}
\usepackage{parskip}
\usepackage{float}
\newtheorem{theorem}{Theorem}
\newtheorem{assumption}[theorem]{Assumption}
\newtheorem{remark}[theorem]{Remark}
\newtheorem{corollary}[theorem]{Corollary}
\newtheorem{lemma}[theorem]{Lemma}
 
\newcommand{\inComplete}[1]{}

\title{The Effects of Transmission-Rights Pricing \\ on Multi-Stage Electricity Markets}

\author{Erwann de Belloy de Saint-Liénard, Jakub Marecek, Vyacheslav Kungurtsev}

\graphicspath{{Images}}

\begin{document}

\maketitle

\begin{abstract}
Cross-border transmission infrastructure is pivotal in balancing modern power systems, but requires fair allocation of cross-border transmission capacity, possibly via fair pricing thereof. This requirement can be implemented using multi-stage market mechanisms for Physical Transmission Rights (PTRs).
We analyse the related dynamics, and show prisoner's dilemma arises. 
Understanding these dynamics enables the development of novel market-settlement mechanisms to enhance market efficiency and incentivize renewable energy use.
\end{abstract}



\section*{Introduction}


Throughout much of the world, electricity markets have undergone the unbundling \cite{Wilson2002} of generation, transmission, and distribution.
Multi-stage markets (also known as multi-stage settlement systems or decentralized designs) have been introduced \cite{meeus2020evolution} to offer opportunities for risk hedging and executing recourse actions.  
At the same time, unbundled markets in various regions are increasingly being coupled,
such as within the single internal European market for electricity \cite{IME,Chapter23Energylaw}.
Market coupling offers the prospect of reducing prices and their volatility \cite{LAM201880},
by
promoting competition, increasing liquidity, and improving integration of renewable electricity production,
but requires the harmonization of electricity exchange systems, systems for the allocation of cross-border transmission capacity \cite{congestionreview},
and ancillary services.
In coupled (multi-area \cite{10379146}) markets, buyers and sellers of electricity and ancillary services
make use of physical cross-border transmission capacity \cite{congestionreview}, 
limited by constraints on grid stability.
Still, there may be only a few large electricity producers in each zone of the coupled market,
and the cross-border transmission capacity is often only a small fraction of the peak demand in the adjacent zones. 
This makes the mathematical modeling of the effects of cross-border transmission on electricity markets
important and rather challenging. 

Let us exemplify the developments on the current situation in Europe, within the so-called capacity allocation and congestion management (CACM, \cite{CACM}) 
model. 
There are a number of market-coupling systems operational. 
In day-ahead markets, 
there are 
Price Coupling of Regions (PCR) and 
Flow-Based Market Coupling (FBMC, \cite{KRISTIANSEN2020100444,CORONA20221768,Felling2023}) 
utilized within Single Day-Ahead Coupling (SDAC).
In intraday markets,
Cross-border Intraday Coupling (XBID) 
has been used within Single Intraday Coupling (SIDC)  
since 2018. 
Constraints on the interconnectors to ensure grid stability are being computed using the Pan-European Hybrid Electricity Market Integration Algorithm (EUPHEMIA, \cite{6861275})
for almost a decade.
Yet, the utilisation of some of the interconnectors, such as on the Hungarian border \cite[p. 3]{EFET}, is as low as 20\%.
Harmonization of balancing markets is proving more challenging still: markets for
primary reserves (frequency containment reserve, FCR) have been fully coupled within the continental Europe synchronous area, 
while the coupling of markets for secondary reserves (automatic Frequency Restoration Reserve, aFRR, \cite{BACKER2023107124})
has been deployed only across selected markets in central Europe so far.
In most smaller countries, it has not proven efficient to break down the former
monopolistic electricity producers into more than two or three market participants. 

Motivated by these developments, we study the dynamics of a collection of interacting duopolies,
and the effects of the cross-border capacity and its pricing on the behavior of market participants in Nash equilibrium.
This builds on the work of Joskow and Tirole \cite{Tirole2000}, who studied the influence of the cross-border capacity market on the day-ahead and intraday markets regarding the market power held by market participants, with a particular focus on how transmission rights can reinforce the market power held by participants. 
This also builds upon the earlier work of Allaz and Villa \cite{cournot},
who consider Cournot-competition models of how day-ahead trading affects producers' and consumers' welfare,
advocating the use of multi-stage markets.
We present:
\begin{itemize}
\item result on Cournot competition and the day-ahead market, where the generators do not have the same marginal cost of production,
\item results on Cournot competition in a model with two duopolies interacting with each other, discussing why in some settlement mechanisms, offering lower prices than the spot market may maximize social welfare.
\item an analysis of the allocation of cross-border transmission capacity in both the primary and secondary markets.
\end{itemize}
Our results illuminate the pivotal role of cross-border transmission and demand forecasts in electricity-market regulation and highlight the potential benefits and challenges of implementing specific pricing policies and market designs to foster renewable energies and enhance overall social welfare.


\section{Cournot competition in a two-stages market with asymmetric duopolies}
\subsection{Introduction to the model}
The presence of day-ahead markets has become increasingly important in shaping strategic decisions within economic markets. 
Allaz and Villa \cite{cournot} investigate how day-ahead trading affects producer and consumer welfare.
Their model involves a Cournot duopoly with symmetric producers (they have the same marginal cost of production) 
who engage in day-ahead market transactions $N$ periods before production in the spot market.
The study explores the impact of the number of day-ahead trading periods, $N$, on market outcomes. 
As $N$ approaches infinity, indicating extensive day-ahead trading, the equilibrium outcome tends towards a competitive solution. 
Social welfare is increased as $N$ tends to infinity, which advocates for the use of multi-stage markets.

Here, we start by extending the model of Allaz and Villa \cite{cournot} to a two-stage market where generators do not have the same cost function and prices are elastic.
In the model, we have two generators (indexed by 1 and 2) and the following variables: 
\begin{itemize}
\item[$\bullet$] $x_1$,$x_2$: the production of generator $i = 1, 2$ 
\item[$\bullet$] $f_1$,$f_2$: their sales in the day-ahead market.
\item[$\bullet$] p: the price of electricity in the day-ahead market
\item[$\bullet$] $c_1(x) = \alpha_1 x$, $c_2(x)= \alpha_2 x$: the cost functions
\item[$\bullet$] $q(z)=D - e z$: the inverse demand function (i.e. the market price of electricity). We assume that $D$ is greater than the marginal costs $\alpha_i$, and the coefficient $e$ represents the demand elasticity of the market
\item[$\bullet$] $p(z)= D^{day-ahead}- ez$: the inverse demand function in the day-ahead market (i.e. the market price of electricity in the day-ahead market).
\end{itemize}
 Throughout, we solve the equilibrium system analytically using  backward induction.

\subsection{Generation game}
\label{B-Generation game}
We denote the profit of the generators in the spot market by $U^R_1(x_1;f_1,f_2, x_2)$ (respectively $U^R_2(x_2;f_1,f_2, x_1)$) for generator 1 (2). Given that generator 1 has sold $f_1$ (respectively $f_2$) in the day-ahead market, it can only sell $x_1 - f_1$ (respectively $x_2 - f_2$) in the spot market at the price $q(x_1 + x_2)$, and it will have to produce $x_1$:
\newline
\begin{align*}
        U^R_1(x_1;f_1,f_2, x_2) &=  q(x_1 + x_2)(x_1 - f_1) - \alpha_1 x_1 \\ & = (D - e(x_1 + x_2))(x_1 - f_1) - \alpha_1 x_1 \\
        U^R_2(x_2;f_1,f_2, x_1) &=q(x_1 + x_2)(x_2 - f_2) - \alpha_2 x_2 \\&= (D - e(x_1 + x_2))(x_2 - f_2) - \alpha_2 x_2 
        \end{align*}
\newline We can see that $U^R_1$ is strongly concave with respect to $x_1$, and similarly $U^R_2$ with respect to $x_2$, implying that any stationary point is a local maximum. 

Consider that each generator maximizes their corresponding profit function with respect to the amount of produced energy. A Nash equilibrium corresponds to a set of quantities $x^*_1,x^*_2$ such that $x^*_1$ maximizes $U^R_1$ given $x_2^*$ and $x^*_2$ maximizes $U^R_2$ given $x^*_1$. The gradient of a strongly concave function is strongly monotone, and so there exists a $\gamma>0$ sufficiently small such that $I-\gamma\begin{pmatrix} \nabla_{x_1} U^R_1(\cdot;f_1,f_2,\cdot) \\
\nabla_{x_2} U^R_2(\cdot;f_1,f_2,\cdot) \end{pmatrix}$ is contractive, and thus a Nash equilibrium exists by the Banach-Picard fixed point theorem. 

Computing the partial derivatives with respect to $\{x_i\}_{i \in \{1,2\}}$ setting them equal to zero, and rearranging the expression, we get the response functions $x_1(x_2)$ and $x_2(x_1)$. Then, we compute the Nash equilibrium of this game (i.e. solve the system of equations given by these two response functions). It gives us:
\begin{equation*}
   x_1(f_1,f_2) =\frac{\frac{1}{e}(D - 2\alpha_1 + \alpha_2) + 2f_1 - f_2}{3} \\
\end{equation*}
\begin{equation}
\label{eq:a}
x_2(f_1,f_2) =\frac{\frac{1}{e}(D - 2\alpha_2 + \alpha_1) + 2f_2 - f_1}{3} \\
\end{equation}
\begin{equation*}
q(f_1,f_2) =\frac{\frac{1}{e}(D + \alpha_1 + \alpha_2) - f_1 - f_2}{3} 
\end{equation*}
\begin{proof}
    see Appendix \ref{app: 1-B}
\end{proof}
Given this results, we can compute the values of $U_1^R$ and $U_2^R$ given $f_1$ and $f_2$.
\newline Here, we can see that the quantities sold in the spot market ($x_1(f_1,f_2) - f_1$ and $x_2(f_1,f_2) - f_2$) only depend on the overall day-ahead sales and not on what each generator has already sold. As we expected, the more is sold on the day-ahead market, the less is sold in the spot market.

\subsection{Day-ahead market equilibrium}

Now we can compute the day-ahead market equilibrium, we have to maximize the profit in the day-ahead market which is $U^D_1(f_1;f_2)$ (respectively $U^D_2(f_2;f_1)$) for generator 1 (2). In this setting, ones also considers electricity sold in the day-ahead market at price $p(f_1,f_2)$:
\newline
 \begin{align*} &U^D_1(f_1;f_2) =U^R_1(f_1;f_2) + p(f_1,f_2)f_1\\ &= (q(f_1,f_2) - \alpha_1 )x_1(f_1,f_2)  + (p(f_1,f_2) - q(f_1,f_2))f_1 \\ &U^D_2(f_2;f_1) =U^R_2(f_2;f_1) + p(f_1,f_2)f_2\\ &= (q(f_1,f_2)- \alpha_2)x_2(f_1,f_2)+ (p(f_1,f_2) - q(f_1,f_2))f_2 \end{align*}

Here the terms $(p(f_1,f_2) - q(f_1,f_2))f_1$ and $(p(f_1,f_2) - q(f_1,f_2))f_2$ are the arbitrageurs profits. It is the additional profit that arbitrageurs can make by trading in the day-ahead market instead of the spot market. To simplify the resolution, we will add the following assumption.

\begin{assumption}[No-arbitrage Property]
\label{Assumption1:"No-arbitrageProperty"}
\begin{center}
    $p(f_1,f_2) = q(f_1,f_2) \Leftrightarrow D = D^{day-ahead}$ 
\end{center}
\end{assumption}

This property states that market participants have perfect foresight of the future and one cannot make any profit by deliberately withholding supply in the spot market to trade in the day-ahead market. Note that in practice this property is enforced by regulation authorities.
\newline As a consequence we can simplify the day-ahead profit: 
\begin{align*}
    & U^D_1(f_1;f_2) = (q(f_1,f_2) - \alpha_1)x_1(f_1,f_2)\\
    & U^D_2(f_2;f_1) = (q(f_1,f_2) - \alpha_1)x_2(f_1,f_2)
\end{align*}
\newline We obtain expressions for response functions $f_1(f_2)$ and $f_2(f_1)$ when we substitute the expressions in \ref{eq:a} and solve for the maximum profit for each generator i, $ i \in \{1,2\} $ with respect to $f_i$:
\hfill \break
\newline $
\left\{
    \begin{array}{ll}
        \dfrac{\partial U^D_1(f_1;f_2)}{\partial f_1}=0 \Leftrightarrow f_1 = \frac{1}{4}(\frac{1}{e}(D - 2 \alpha_1 + \alpha_2) - f_2)\\
        \dfrac{\partial U^D_2(f_2;f_1)}{\partial f_2}=0 \Leftrightarrow f_2 = \frac{1}{4}(\frac{1}{e}(D - 2 \alpha_2 + \alpha_1) - f_1)
    \end{array}
\right.
$
\hfill \break
\newline We can then compute the Nash equilibrium of this system which gives us:
\begin{equation*}
    f_1=\frac{D - 3\alpha_1 + 2\alpha_2 }{5e} \,
\qquad
f_2=\frac{D - 3\alpha_2 + 2\alpha_1 }{5e} \\
\end{equation*}
\begin{equation}
\label{eq:b}
x_1=\frac{2(D - 3\alpha_1 + 2\alpha_2 )}{5e} \,
\qquad
x_2 =\frac{2(D - 3\alpha_2 + 2\alpha_1 )}{5e} \\
\end{equation}
\begin{equation*}
q=\frac{D +2( \alpha_1 + \alpha_2)}{5e} 
\end{equation*}

  If we take $\alpha_1 = \alpha_2$, then we get the same results as \cite{cournot}. 
In this model, we introduced an elasticity coefficient, and interestingly, we observed that its value had no impact on the distribution of sales between the day-ahead and spot markets. However, the introduction of the elasticity coefficient presents one novel observation relative to \cite{cournot}. By considering the expression for $q$ we quickly see that  an inelastic market leads to significant price increases. Recognizing this potential issue, Europe has implemented a price cap mechanism to prevent soaring electricity prices, considering that European energy demand predominantly exhibits inelasticity.
In Europe, the long-run price elasticity is estimated between 0.53 and 0.56 \cite{elasticity}, so we can say that the electricity market is relatively inelastic.
\newline The marginal cost of production varies greatly depending on the type of energy. It is close to zero for renewable energies such as wind power\inComplete{(in this case, we would be more inclined to consider an opportunity cost)}, whereas it is much higher for thermal power plants, for example.

\section{Model 1: A two-stage model with two markets which are duopolies}
\label{Model 1: A two-stage model with two markets which are duopolies}
\subsection{Description of the model}
Considering a number of smaller electricity markets are dominated by duopolies, 
we will model all countries as duopolies. Sales between countries are possible thanks to cross-border transmission capacities.
Therefore, to understand the management of the transmission infrastructure and the allocation of transmission capacity between countries, we will model transmission lines between two countries as four capacities dispatched to each of the generators in the duopolies.
 The maximum amount of electricity that can be sold on the foreign market is limited for each generator by its transmission capacity $\{K_i\}_{i \in I}$. We also add a congestion cost $\eta $ which is an additional cost that generators have to pay to export their electricity to a given market.
\newline On the demand side, we will use elastic demand with a price cap that depends on the market situation.
  In this first model, we want to study the behavior of a multi-stage electricity market, so we will start with a two-stage market with a spot market and a day-ahead market. The spot market is the market that operates just before the actual delivery of electricity.
The day-ahead market operates 24 hours before the spot market, based on a forecast of demand on the spot market.
  We will represent the spot market as a set of scenarios S and their probability $p_s$ which represent all the possible values of the maximum price $D_s$ in the spot market. The idea behind that is that we will be able to optimize the production in the day-ahead market regarding all the possible values of the demand in the spot market weighted by their probability, targeting the expected profit given the uncertainty.

   Figure~\ref{fig:explo_c} presents a diagram of the model with two duopolies. The structure of the market will remain the same, while the details are discussed next.
\begin{figure}[tb]
    \centering
    \includegraphics[width = 0.45\textwidth]{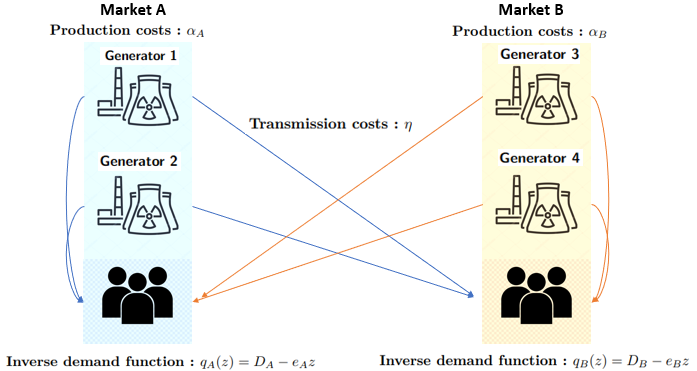}
    \caption{A schematic illustration of two markets with a duopoly of two generators in each market.}
    \label{fig:explo_c}
\end{figure}

 Table~\ref{tab:notations} lists variables appearing in the market settlement model for the four generators across markets A and B:
 \begin{table}
\begin{tabular}{ | m{3cm} | m{5cm}| } 
  \toprule
  \textbf{Symbol} & \textbf{Definition}  \\ [0.5ex] 
 \midrule
 $s \in S$ & The scenario $s$ that can occur on the spot market, and their set $S$    \\ 
 $I = [\![1;4]\!]$ & The set of indices of generators   \\ 
 $q_A^s(z) = D_A^s - e_A z$ & The inverse demand function for market A on the spot market in scenario s \\ 
  $q_A(z) = D_A^{SO} - e_A z$ & The inverse demand function for market A on the day-ahead market\\ 
 $\forall i \in [\![1;4]\!],x_i$ & The overall production of each generator  \\
 $\alpha_A$ (respectively $\alpha_B$)  & The marginal cost of production of generators from market A (respectively B)  \\
 $\eta$ & The marginal transmission cost   \\
 $f_i$ (respectively $g_i$) & The day-ahead sales in the market A (respectively B) for each generator  \\
$ y_i^s$ (respectively $z_i^s$) & Their sales in the spot market A (respectively B) in the scenario s \\ 
 $y_i = \sum_s p_s y_i^s$ ( $z_i$) & Their expected sales in the spot market A (respectively B) \\ 
  $ K_i$ & The transmission capacity of each generator \\
  $D_A = \sum_{s \in S}p_s D_A^s$ & The average value of $D_A^s$ in market A \\ 
 $\beta_A^s = D_A^{SO} - D_A^s$ & The difference between prices in the day-ahead and in the spot market in scenario s\\ 
 $\beta_A = D_A^{SO} - D_A$ & The average difference between prices in the day-ahead and in the spot market\\ 
 \bottomrule
\end{tabular}
 \caption{A table of notation utilized in the market settlement model for the four generators across markets A and B.}
 \label{tab:notations}
\end{table}
 We assume that the available transmission capacity $K$ between market A and B has been dispatched prior to the day-ahead market. This means that when the day-ahead (and therefore the spot) market is settled, generators already know their transmission capacity.

In this first model, we will see how to choose the value of $D_A^{SO}$ on the day-ahead market to optimize social welfare.
  There, we will start by studying the quantities sold on the day-ahead and spot markets given this settlement of the market and then see how one can optimize the value of $D_A^{SO}$. 

\begin{remark}
    The price (inverse demand function) in the day-ahead market is the difference between the maximum price $D_A^{SO}$ and the realized production (in both stages) while the price in the spot market is the difference between the actual maximum price $D_A^s$ that consumers are willing to pay and the realized production.
\end{remark}
We note that market A and B are independent (producers do not need to arbitrage between these two markets since they can produce an infinite quantity of electricity). By symmetry, we can clear only one market and extend the results to the other. Consequently, we will only focus on market A as we will have symmetric results in market B.
\inComplete{We mention that this will not be the case if we choose to limit the maximum amount of electricity produced by each generator, as these constraints will bind the amount produced in each market together, but this is outside the scope of this model.}
The solution of this and other models will follow the same idea, namely "backward induction". We will start by clearing the spot market based on the results of the day-ahead market, then clear the day-ahead market. This method is taken from \cite{cournot}.
\subsection{Resolution of the model}
\subsubsection*{1) Generation game}
Each generator wants to maximize its profit from their sales in the spot market (i.e. $y_j^s$). Here we maximize the difference of the revenue, with price $q_A^s(z)$ with $z= \sum_{i=1}^4 y_i^s + f_i $, the overall sales in market A. The cost of production is $\alpha_j$ and cost of transmission $\eta$ if the quantity is produced in a foreign market. In all, for generator $j \in [\![1;4]\!]$:
\begin{align*}
    & \forall j \in \{1,2\},\forall s \in S,\\
    & U^{Rs}_j(\{f_i,y_i^s\}_{i \in I}, \eta)=q_A^s[\sum_{i=1}^4 y_i^s + f_i] y_j^s- c_j(y_j^s+f_j)\\
    & = y_j^s (D_A^s - e_A \sum_{i = 1}^4(y_i^s + f_i))- \alpha_A (y_j^s + f_j)\\
    & \forall j \in \{3,4\},\forall s \in S,\\
    &U^{Rs}_j(\{f_i,y_i^s\}_{i \in I}, \eta)= q_A^s[\sum_{i=1}^4 y_i^s + f_i] y_j^s - (c_j+ \eta)(y_j^s+f_j) \\ & = y_j^s (D_A^s - e_A \sum_{i = 1}^4(y_i^s + f_i))- (\alpha_A + \eta) (y_j^s + f_j)
\end{align*}

In the model, we only have constraints on how much can be exported by each generator (i.e. generators who export from market B to market A and from market A to market B). As a consequence, when studying price and quantities sold to consumers from market A by each generator, only values regarding constraints on exports from market B to market A will appear (i.e. generators 3 and 4). Constraints on generators from market A who export to market B will only appear when we optimize the sales in market B (while constraints on generator 3 and 4 would not appear). The transmission constraints regarding sales in the spot market A are :
\setlength{\abovedisplayskip}{10pt}
\setlength{\belowdisplayskip}{10pt}
\begin{align*}
    &\forall s \in S, F^{Rs}_3(f_3,y_3^s) = f_3 + y_3^s - K_3 \le 0\\
    &\forall s \in S, F^{Rs}_4(f_4,y_4^s) = f_4 + y_4^s - K_4 \le 0
\end{align*}
\label{lagrange}
The Lagrangian of the problem is, in the spot market A (with $\forall i \in \{3,4\},\lambda_j^{0s} \ge 0$ the Lagrange multipliers associated with the constraint $F_j^{Rs}(z)$ where $z=\{f_i, y_i^s\}_{i \in I}$) :
\begin{align*}
    &\forall s \in S,\forall j \in \{1,2\},\\
    &L_j^{Rs}(\{f_i,y_i^s\}_{i \in I},\eta) = -U^{Rs}_j(z, \eta),\\
    &\forall s \in S,\forall j \in \{3,4\},\\
    &L_j^{Rs}(\{f_i, y_i^s\}_{i \in I},\eta,\lambda^{0s}_j) = -U^{Rs}_j(z, \eta)  + \lambda^{0s}_j F^{Rs}_j(z)
\end{align*}
The transmission constriants are linear. Consider that we shall enforce, in the first stage problem, the constraint $y_3^*\le K_3$. In this case, either $y_3^*=K_3$, and $f_3= 0$ is necessary for feasibility and the problem reduces to being only in terms of $f_4$, or otherwise the constraint can be a strictly feasible, i.e., Slater's constraint qualification holds. With a strongly convex objective, we can conclude that the KKT optimality conditions are necessary for a local minimizer. 

More specifically, we know that $\forall j \in I,\forall s \in S,$ a solution ($y_j^s,\lambda_j^{0s}$) is a saddle point of the Lagrangian $L_j^{Rs}(y_j^s,\lambda_j^{0s})$.
\newline So, we have to minimize these equations in $\{y_j^s\}_{j \in I}$, we solve these equations and get the response functions.
\newline $\forall j \in I,\forall s \in S, x_j^{As}(x^{As}_{i \ne j})$:
\setlength{\abovedisplayskip}{2pt}
\setlength{\belowdisplayskip}{1pt}
\begin{align*}
    &\dfrac{\partial L^{Rs}_j(\{f_i,y_i^s\}_{i \in I},\eta,\lambda^{0s}_j)}{\partial y_j^s} = 0   
\Leftrightarrow \\
    &\left\{
    \begin{array}{ll}
        \forall j \in\{1,2\}, x_j^{As} + \frac{1}{2}\sum_{i \ne j} x_i^{As} = \frac{D_A^s - \alpha_A + e_A f_j}{2 e_A}\\
        \forall j \in\{3,4\}, x_j^{As} + \frac{1}{2}\sum_{i \ne j} x_i^{As} = \frac{D_A^s - \alpha_B - \eta - \lambda^{0s}_j + e_A f_j }{2 e_A}
    \end{array}
\right.
\end{align*}
\begin{proof}
See Appendix \ref{app: 2-B}
\end{proof}
The combination of KKT conditions present a Linear Complementarity Problem (LCP). From standard fixed point arguments based on Kakutani's fixed point theorem, an LCP with a compact feasible region has a solution (see, e.g.~\cite{facchinei2003finite}).

Then, we compute the Nash equilibrium for these equations, by finding the solution of this matrix system $AX_A^s=B_x^s$ with:
\begin{align*}
A &= \begin{bmatrix}
    1 & \frac{1}{2} & \frac{1}{2} & \frac{1}{2} \\
    \frac{1}{2} & 1 & \frac{1}{2} & \frac{1}{2} \\
    \frac{1}{2} & \frac{1}{2} & 1 & \frac{1}{2} \\
    \frac{1}{2} & \frac{1}{2} & \frac{1}{2} & 1 \\
\end{bmatrix},
X_A^s = \begin{bmatrix}
    x_1^{As} \\
    x_2^{As} \\
    x_3^{As} \\
    x_4^{As} \\
\end{bmatrix}, \\
B_x^s &= \frac{1}{2e_A}
\begin{bmatrix}
    D_A^s + e_A f_1 - \alpha_A \\
    D_A^s + e_A f_2 - \alpha_A \\
    D_A^s + e_A f_3 - (\alpha_B + \eta + \lambda^{0s}_3) \\
    D_A^s + e_A f_4 - (\alpha_B + \eta + \lambda^{0s}_4) \\
\end{bmatrix}
\end{align*}
\newline Then, we have the solutions for sales in the spot market A: 
\setlength{\abovedisplayskip}{10pt}
\setlength{\belowdisplayskip}{10pt}
\begin{align*}
    &\forall s \in S,q_A^s(\{f_i\}_{i \in I}) : = C_A^s - \frac{e_A}{5}\sum_{k=1}^4 f_k \\
    & = \frac{1}{5}(D_A^s + 2 (\alpha_A + \alpha_B + \eta)+ \lambda^{0s}_3 + \lambda^{0s}_4 - e_A\sum_{k=1}^4 f_k) \\
    & \forall s \in S,\forall j \in \{1,2\},  y_j^s : = r_j^s - \frac{1}{5}\sum_{k=1}^4 f_k \\&
    = \frac{1}{5e_A}(D_A^s - 3 \alpha_A + 2(\eta + \alpha_B) + \lambda^{0s}_3 + \lambda^{0s}_4 - e_A \sum_{k=1}^4 f_k )\\ &
    \forall s \in S,\forall j \in \{3,4\},
    y_j^s : = r_j^s - \frac{1}{5}\sum_{k=1}^4 f_k \\&
    =  \frac{1}{5e_A}(D_A^s - 3 (\eta + \alpha_B) + 2 \alpha_A  - 4  \lambda^{0s}_j + \lambda^{0s}_{-j}  - e_A\sum_{k=1}^4 f_k )\\
\end{align*}
\inComplete{
We have symmetric results for market B (generators who export their production are now generators 1 and 2):
\newline $\forall s \in S,\forall j \in \{1,2\},$
\newline $\begin{aligned}[t]
    z_j^s &= \frac{1}{5e_B}( D_B^s - 3(\eta + \alpha_A) + 2 \alpha_B- 4  \lambda^{0s}_j + \lambda^{0s}_{-j}  - e_B\sum_{k=1}^4 g_k) \\&: = s_j^s - \frac{1}{5}\sum_{k=1}^4 g_k
\end{aligned}\\$
\newline
$\forall s \in S,\forall j \in \{3,4\},$
\newline$
    \begin{aligned}[t]
    z_j^s& =  \frac{1}{5e_B}(D_B^s - 3 \alpha_B + 2 (\eta + \alpha_A)  + \lambda^{0s}_1 + \lambda^{0s}_2 - e_B\sum_{k=1}^4 g_k)\\& : = s_j^s - \frac{1}{5}\sum_{k=1}^4 g_k 
\end{aligned}$
    \newline $
    \forall s \in S,$
\newline $
\begin{aligned}[t]
        q_B^s(\{g_i\}_{i \in I}) &= \frac{1}{5}(D_B^s + 2 (\alpha_A + \alpha_B + \eta) + \lambda^{0s}_1 + \lambda^{0s}_2  - e_B\sum_{k=1}^4 g_k)\\ &: = C_B^s - \frac{e_B}{5}\sum_{k=1}^4 g_k 
\end{aligned}  $}
We can see that the greater the quantity sold on the day-ahead market is, the less producers will sell on the spot market.
Then, the situation for producers regarding prices in the spot market is as if they have not sold anything in the day-ahead market and that the price cap is just $D_A^s - e_A\sum_{k=1}^4 f_k$.
\subsection*{2) Day-ahead market}
\label{arbitrage}
At time 1 (day-ahead market), a selected demand function based on a market-demand forecast will be utilized.
\newline The underlying analysis is the idea that $D_A^s$ represents demand, and the higher $D_A^s$, the higher the demand.
\newline The generation game will not be modified as generators will observe a scenario $s \in S$ for the demand function and have the day-ahead sales as an input.
 The $D_A^{SO}$ term is chosen to optimize social welfare. 
That is: one chooses $D_A^{SO}$, such that it maximizes the sum of the consumer and producer surpluses:
\begin{align*}
    & z(\beta_A)=  \sum_{s\in S} p_s [ \int_{0}^{x_A^s} D_A^s - e_A xdx] - \alpha_A x_{AA} \\ & - (\alpha_B + \eta) x_{BA}- \beta_A \sum_{i \in I}f_i] 
\end{align*}

Notice that the term $- \beta_A \sum_{i \in I}f_i $. This is the additional cost or profit of the arbitrageurs profit.

\begin{figure*}
    \centering
    \begin{theorem}[Optimal profit of arbitrageurs] \label{optimalprofit}
The arbitrageurs profit chosen to optimize social welfare $\beta_A$ is:
\begin{equation}
\frac{D_A(e_A -12) + (\alpha_A + \alpha_B + \eta)(8e_a - 4) + (\lambda_3^0 + \lambda_4^0)(4e_A +3)+ (\lambda_3^1 + \lambda_4^1)(\frac{3}{5}e_A+3) }{4 e_A + 40}
\end{equation}
\end{theorem}
    \label{fig:enter-label}
\end{figure*}

\begin{proof}
We compute the partial derivatives regarding $\beta_A$, see Appendix \ref{app: 2-B arbitrageurs profit} 
\end{proof}
The main difference between the various European markets is the elasticity of demand. Here, we can see the effect of the choice of $D_A^{SO}$.
\bigskip \newline \textbf{Case 1:} if the demand is almost inelastic, $e_A \to 0$, we have: 
\begin{equation*}
        D_A^{SO}= \frac{28 D_A -4 (\alpha_A + \alpha_B + \eta) +3 (\lambda_3^0 + \lambda_4^0+ \lambda_3^1 + \lambda_4^1) }{40}
\end{equation*}
\newline In this case, one can choose a price that is lower in the day-ahead market than in the spot market. 
\bigskip \newline \textbf{Case 2:} if the demand is very elastic $e_A \to +\infty$, we have $D_A^{SO}$: 
\begin{equation*}
         \frac{5D_A + 8(\alpha_A + \alpha_B + \eta) + 4(\lambda_3^0 + \lambda_4^0)+ \frac{3}{5}(\lambda_3^1 + \lambda_4^1)}{4}
\end{equation*}
\newline In this case, one chooses a price that is higher in the day-ahead market than in the spot market.
\bigskip \newline \textbf{Case 3:} if $e_A=1$, then $D_A^{SO}$ is:
\begin{equation*}
         \frac{33D_A + 4(\alpha_A + \alpha_B + \eta) +7( \lambda_3^0 + \lambda_4^0)+ \frac{18}{5}(\lambda_3^1 + \lambda_4^1) }{44}
\end{equation*}
\newline In this case, which market has a higher price will depend on the values of $D_A, \alpha_A, \alpha_B, \eta$.
  
  The main result that follows from this model without the ``no-arbitrage property'' is the ``prisoner's dilemma'' \cite{tucker1959contributions} that can arise between producers: A demand on the day-ahead market can be lower than the expected demand on the spot market, as this would lead to higher overall social welfare in some cases.
  At first, it might seem preferable for producers to cooperate and not enter the day-ahead market, as prices will be higher on average than on the day-ahead market. But if they do not enter the market and their competitors do, this will be very detrimental to them.
\newline In the end, they will prefer to secure their profit by entering in the day-ahead market, even if it would be better for them to cooperate.
\newline This results from Nash equilibrium strategies. A set of strategies which is a Nash equilibrium is an equilibrium where no producers have an interest to unilaterally deviate from their strategy.
\newline To illustrate this phenomenon, we will look at an example where the markets are almost inelastic and so the arbitrageurs' profit is negative ($\beta_A \le 0$).
If only generator 1 enters the day-ahead market, given the results above, we have the profit values for each generator, so let's concentrate on generators 1 and 2 and their profit (respectively $\Pi_1$ and $\Pi_2$):
\newline
\begin{align*}
    &\Pi_1 = (r_A + \frac{4}{5}f_1)(C_A - \alpha_A - \frac{1}{5}f_1) + \beta_A f_1 \\
    &\Pi_2 = (r_A - \frac{1}{5}f_1)(C_A - \alpha_A - \frac{1}{5}f_1)
\end{align*}
\newline Then, $\Pi_2 = \Pi_1 + f_1 (C_A - \alpha_A - \frac{1}{5}f_1 + \beta_A)=  \Pi_1 + f_1 (q_A + \beta_A)\ge \Pi_1\\$ as $q_A \ge -\beta_A = q_A - p_A$ (because we can't have a negative price in the day-ahead market)
\newline Therefore, cooperating is not a Nash equilibrium in this game, as producers will have an incentive to deviate from their position if everyone only enters the spot market. As a result, everyone will end up entering the day-ahead market even if prices are lower.
\subsection{An analysis of the first model}
In their paper \cite{cournot}, Allaz and Villa make a strong assumption, namely that producers and arbitrageurs have a perfect foresight of what will happen in the future. Consequently, we have the ``no-arbitrage'' hypothesis, i.e., the idea that producers already know, when they sell on the day-ahead market, what demand will be on the spot market. This is particularly important when it comes to optimizing profit according to demand and capacity constraints on the spot market (producers do not want to buy or sell too much given the realization of demand and the transmission capacities they will obtain).
  In the electricity market, the situation is quite different: producers do not know what will happen in the future, and only a forecast of scenarios that may occur is available. That is why we are removing the ``no-arbitrage''  assumption from Cournot's paper \cite{cournot}. When trading on the day-ahead market, different strategies can be chosen with regard to scenarios and whether capacity constraints will be reached. Our focus is on maximizing expected profit, but we could have examined profit risk, for example.

When we looked at the optimal profit for arbitrageurs, we saw that one may wish to offer lower prices in the day-ahead market, because producers will not cooperate in the market, they will enter the day-ahead market even if they know that on average their profit will be lower.
\newline While this may seem desirable, since it will increase social welfare, in this situation we would be reducing the share of day-ahead sales. In the electricity market, it is very important to plan and launch generation long before the actual delivery of electricity. This parameter is called the activation time of each generator. Consequently, all these parameters should be taken into account when estimating demand for the day-ahead market, and not just social welfare as in the model.

  The fact that markets A and B export to each other can be inefficient in terms of social welfare. As a result, it increases transport costs and the risk of transmission line congestion. One solution could be to trade the quantity they export between themselves, in order to reduce network congestion and transmission costs.
\newline If we denote $x_{AB}$ ($x_{BA}$) the quantity exported from A to B (from B to A), we can reduce the overall amount by only exporting the quantity $z= \min\{x_{AB},x_{BA}\}$ and pay $z \times p_A$ to generators from area A and pay $z \times p_B$ to generators from area B.
  A drawback of such policy is that it reduces the number of producers which sell to a given market and as a consequence it increases uncertainty over the actual delivery of electricity in the spot market as we rely on fewer generators. Some recent examples such as Hungary where the price soared up to 1200 Eur/MWh on the $12^{th}$ of July 2023 have shown that a poor forecast of production or demand in a given region can send prices soaring.
  Then, the question is to know how much transmission capacity will be available, as this model does not study the allocation of cross-border transmission capacity. This will be the main focus of later models, where we will study hedging strategies against congestion due to lack of transmission capacity.

\section{Model 2: Physical Transmission Rights (PTRs)}
\subsection{Description}
\label{sec: PTRs}
\subsubsection*{1) Allocation of PTRs in Europe}
In Europe, the Joint Allocation Office (JAO) auctions cross-border transmission capacity rights.
Auctions are held in a day-ahead market with the following rules: 
\begin{itemize}
    \item The available capacity $K$ between country A and B is calculated by the JAO.
    \item Each market participant can submit a bid or a set of bids in quantities and prices
    \item Winning bids will be those with the highest price within the available capacity K. As a consequence, some bids can be partially or fully accepted.
    \item The price to be paid by market participants will be the lowest price within the accepted bids. If the sum of quantities for the bids offered is below K, then the price will be 0 and all bids are accepted.
    \item The market implements the policy Use It Or Lose It (UIOLI): this means that if a generator does not use its transmission capacities, then he will lose it and this can be sold again to another generator.    
\end{itemize}
On the European market, there is no secondary market for transmission rights, 
which would be useful when market participants wish to adjust their capacities given the realization of the demand in the electricity market.
\subsubsection*{2) Description of the Model}
Figure~\ref{fig:2} shows the timeline of the model with auctioned capacity constraints for each generator  
primary and secondary markets in PTR:
\begin{figure}[tb]
    \centering
    \includegraphics[width = 0.45\textwidth]{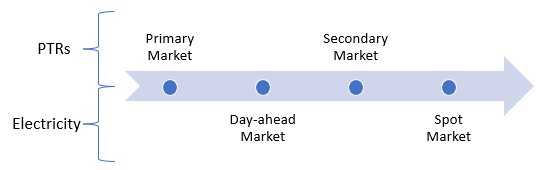}
    \caption{Timeline of the second model}
    \label{fig:2}
\end{figure}
\newline We use the following notations:
\begin{itemize}
    \item[$\bullet$] $K_i^p$, the capacity bought in the primary market by generator i.
    \item[$\bullet$] $K_i^s$, The capacity bought (or sold) in the secondary market by generator i. $K_i^s$ can be either positive or negative.
    \item[$\bullet$] $K_i = K_i^p+ K_i^s$ represents the number of PTRs (Physical Transmission Rights) held by a generator.
    \item[$\bullet$] $\sum_{i=1}^4 K_i \le K$, the sum of the PTRs hold by generators has to be below K.
\end{itemize}
In this model, each generator has its own capacity constraints $K_i^p$ and $K_i =K_i^s + K_i^p$ that it can use to export in the day-ahead and in the spot market. We will also introduce a secondary market for capacities run just before the spot market where generators can trade transmission capacities between each other.
In this model, we keep the assumption of the  ``no-arbitrage'' property for more simplicity (i.e. $D_A^{SO} = D_A$).

These PTRs give the right to sell in the foreign market.
\newline We will then discuss several options that we have to run the primary and secondary markets regarding prices, auctions and regulatory policies.
\subsubsection*{3) Adopted approach}
As we did in the electricity market, we will start by clearing the secondary market given what has been bought by each generator in the primary market for PTRs and what has been sold in the day-ahead market for electricity and then focus on the primary market.
The approach used in the primary market will be quite different as generators have to place bids (i.e., choose a quantity and a price) knowing the overall capacity $K$ available in the capacity market.
\newline In the secondary market, the generators will observe the real demand realized in the spot market. So, they will trade capacities such that they increase their overall profit in the scenario $s \in S$.
\newline
As a consequence, there is a redispatching of $K$ in the spot market after clearing the primary market for capacities and the day-ahead market for electricity (i.e. given a set $K_i^{p}$, what will be the final values $K_i^{s}$ if we run a secondary market between generators when they face a scenario $s \in S$).
\newline We have this identity: $\sum_i K_i^{s} \le \sum_{i=1}^2 (K_i^{p} - g_i) + \sum_{i=3}^4 (K_i^{p} - f_i)$.
  The amount that can be traded in the secondary market is below what has been bought in the primary market and not used yet in the day-ahead market.
Generators must maintain at least a capacity $K_i$ that is above their day-ahead sales in the foreign market.
 The only change with the first model of the electricity market is in the constraints that will be changed in the day-ahead and in the spot market. Now, the values of $K_i$ will vary over time, and so will the values of the optimal sales in each market.
\newline We maximize the profit in markets A and B which is the difference between the revenue generated by the sales in the spot market and the cost of production. This gives us the following equations for generator 1 (now instead of model 1, we optimize in both markets A and B at the same time):
\begin{itemize}
\item [$\bullet$] For the spot market (generation game):
$$
\left\{
    \begin{array}{ll}

        \forall s \in S, U^{Rs}_1(K_i)  = (D^s_A - e_A x_A^s )y_1^s - \alpha_A (y_1^s + f_1) \\ + (D^s_B - e_B x_B)z_1^s -( \alpha_A + \eta)(z_1^s + g_1)        \\
        s.t. z_1^s + g_1\le K_1^{p} + K_1^{s} = K_1
    \end{array}
\right.
$$
\item [$\bullet$] For the day-ahead market:
    $$
\left\{
    \begin{array}{ll}
        U^D_1(K_i)  = (D_A - e_A x_A - \alpha_A)(y_1 + f_1)\\ + (D_B - e_B x_B - \alpha_A - \eta)(z_1 + g_1)\\
        s.t. g_1\le K_1^{p}
    \end{array}
\right.
$$
\end{itemize}

We can expect that if the realization of the demand in the spot market is below the expectations in the market A and above the expectations in market B, then generators from B would be willing to sell some of their capacities to generators from market A.
\newline Here, the primary market is run before the day-ahead market and the secondary market is run just before the spot market.
A generator wants to buy adding capacity in the secondary market (respectively primary market) if and only if it increases its profit (respectively expected profit).

\subsection{Resolution of the model}
\subsubsection*{1) Price bounds in the PTRs secondary market}
We will find inequalities that define the price bounds in the PTRs secondary market.
We use the following notation:
\begin{itemize}
    \item[$\bullet$] $\forall i \in I, \Pi_i^{PTR}(K_i)$ is the optimal value of the maximization problem for generator i on the spot market given the values of transmission capacities and the day-ahead sales.
    \item[$\bullet$] $p^{PTR}_{ij}$ the price that a generator $i$ is ready to pay to buy more transmission capacities, to a generator $j$ or other paries knowing that he would prevent $j$ from buying these capacities.
\end{itemize}
\begin{assumption}
    We make the assumption that if $i$ does not buy a capacity $\delta K$ that is sold, then another generator will buy this capacity instead.
\end{assumption}
This is important: when choosing the maximum price generators are willing to pay, they will take into account the fact that they prevent another market participant from buying this capacity.
\begin{lemma}\label{maxprice}
    A generator $i$ will prefer to buy an adding capacity instead of letting $j$ buying (respectively keeping) in the primary market (respectively secondary market) if and only if:
    \newline $\forall i,j \in I, \forall t \in \{0,1\},$
\begin{equation}
    p^{PTR}_{ij} \le \dfrac{\partial \Pi_i^{PTR}(K_k)}{\partial K_i} - \dfrac{\partial \Pi_i^{PTR}(K_k)}{\partial K_j}
\end{equation}
\end{lemma}
\begin{proof}
    A generator wants to buy an adding capacity $\delta K$ if: 
    \begin{equation*}
        \dfrac{\partial \Pi_i^{PTR}(K_k)}{\partial K_i}\delta K - \dfrac{\partial \Pi_i^{PTR}(K_k)}{\partial K_j}\delta K - p^{PTR}_{ij} \delta K \ge 0
    \end{equation*}
The capacity of the generator $i$ will increase by $\delta K$ and that of the generator $j$ will decrease $\delta K$.
\end{proof}
\begin{lemma}\label{minprice}
    In the secondary market, a generator $j$ would like to sell a capacity $\delta$K to a generator $i$ if and only if:
\begin{equation}
    p^{PTR}_{ij} \ge \dfrac{\partial \Pi_j^{PTR}(K_k)}{\partial K_j} - \dfrac{\partial \Pi_j^{PTR}(K_k)}{\partial K_i}
\end{equation}
\end{lemma}
\begin{proof}
    A generator $j$ wants to sell an amount $\delta K$ to generator $j$ if:
    \begin{equation*}
        p^{PTR}_{ij} \delta K - \dfrac{\partial \Pi_j^{PTR}(K_k)}{\partial K_j}\delta K + \dfrac{\partial \Pi_j^{PTR}(K_k)}{\partial K_i}\delta K \ge 0
    \end{equation*}
\newline The capacity of the generator $i$ will be increased by $\delta K$ and the one from generator $j$ will be decreased by $\delta K$
\end{proof}
\begin{theorem}
    The price for PTRs between generator $i$ (the seller) and $j$ (the buyer) will be within these limits in the secondary market:
\begin{equation}
    \dfrac{\partial \Pi_i^{PTR}(K_k)}{\partial K_i} - \dfrac{\partial \Pi_i^{PTR}(K_k)}{\partial K_j} \le p^{PTR}_{ij} 
\end{equation}
\begin{equation}
    p^{PTR}_{ij} \le \dfrac{\partial \Pi_j^{PTR}(K_k)}{\partial K_j} - \dfrac{\partial \Pi_j^{PTR}(K_k)}{\partial K_i}
\end{equation}
\label{pricebounds}
\end{theorem}
\begin{corollary}
    There will not be any trades between generator $i$ (the seller) and $j$ (the buyer) if:
    \begin{equation}
    \label{impossibletrades}
        \dfrac{\partial \Pi_i^{PTR}(K_k)}{\partial K_i} - \dfrac{\partial \Pi_i^{PTR}(K_k)}{\partial K_j} > \dfrac{\partial \Pi_j^{PTR}(K_k)}{\partial K_j} - \dfrac{\partial \Pi_j^{PTR}(K_k)}{\partial K_i}
    \end{equation}
\end{corollary}
\begin{proof}
    We deduce these bounds from Lemma \ref{maxprice} and Lemma \ref{minprice}.
\end{proof}
\begin{remark}
    When constraint on generator $i \in I$ becomes inactive, partial derivatives regarding $K_i$ are null:
\begin{equation*}
    \forall j \in I, \dfrac{\partial \Pi_i^{PTR}(K_k)}{\partial K_i} = \dfrac{\partial \Pi_j^{PTR}(K_k)}{\partial K_i} =0.
\end{equation*}
\end{remark}

\begin{remark}
There will be trades in the secondary market until that $\forall i,j \in I$:
\begin{itemize}
    \item Either Equation \ref{impossibletrades} is true
    \item Or constraints on generators $i$ and $j$ are inactive.
\end{itemize}
\end{remark}

 We can know the price that every generator will be willing to pay in each market given the realization of the demand in the spot market, if we get the solutions of the maximization problem in each state. This will set the price in the secondary market at each time.
\subsubsection*{1) Resolution of the secondary market for capacities}

We want to know what are the possible allocations of PTRs after running the secondary market.
\newline The goal of this is to see how efficient the implementation of a secondary market could be, and to understand the behaviour of generators in this market.
\newline We want to exhibit the phenomenon of ``rights witholding'' which can be seen by considering the secondary market between two generators from the same area (it is similar if you consider the sell of capacities between generators from different regions). \newline We will continue using notation from Figure~\ref{fig:explo_c}.
\begin{itemize}
    \item Case 1: Capacity constraints are active for generator 3 and not for generator 4
    \item Case 2: Capacity constraints are active for generators 3 and 4
\end{itemize}
\inComplete{
\newline Now, we will study the different possible trades when there are active constraints. We will make the assumption that the constraint on generator 3 is active (we will have symmetric results for the other generators).We assume that 3 is a buyer as his constraints are active but we have symmetric situations where he could be a seller if for example the constraint is active for the other generator. We can also have cases where even if they will not use these capacities, some generators will be willing to buy more capacities if their price is very low to prevent other market participants from buying and using it. Table~\ref{tab-possible} sums up all the possible trades for generator 3 in the secondary market:
\begin{table}
 \begin{tabular}{|c|c|c|c|c|c|}
        \hline 
         Case & 3 trades w/ & 4 active? & 1 active? & 2 active? \\ \hline 
         A.1 & 4 & No & &   \\ \hline
         A.2 & 4 & Yes & &  \\ \hline
         B.1 & 1 & No & No & No   \\ \hline
         B.2 & 1 & No & Yes & No  \\ \hline
         B.3 & 1 & Yes & No & No   \\ \hline
         B.4 & 1 & Yes & Yes & No  \\ \hline
         B.5 & 1 & No & Yes & Yes \\ \hline
         B.6 & 1 & Yes & Yes & Yes  \\ \hline
    \end{tabular}
\caption{The cases considered within the resolution of the secondary market for capacities.}
\label{tab-possible}
\end{table}}
Here, given the KKT conditions of the Lagrangian of the maximization problem, we have that:
\newline $\forall t \in \{1,2\},\forall j \in I,\lambda^R_j.F^R_j(\{f_i,g_i, y_i, z_i\}_{i \in I})=0$ (complementary condition), we can compute the values of Lagrange multipliers ($\lambda^R_j$) in each case:
\begin{itemize}
    \item[$\bullet$]  First case: $\lambda^R_j= 0 $ so we have the same equations as before (the capacity constraint is not active)
    \item[$\bullet$] Second case: $F^R_j(\{f_i,g_i, y_i, z_i\}_{i \in I}) = 0$, the capacity constraint is active
\end{itemize}
We can compute the values of $\lambda^R_j \ge 0$, when there is congestion in the cross-border transmission (by introducing the values of $f_i$,$g_i$,$y_i$,$z_i$) and enumerating which constraints are active or not. It will give us the different possible values for $\lambda^R_j$.
\newline Then, the only problem will be to know which constraints are active and solve the right system of equations in each case to get the values of $\lambda^R_j$.
\newline Regarding trades between generators from the same market, we can deduce the results without computing the values of the profits and its derivatives regarding $\{K_i\}_{i \in I} $.

\textbf{Case 1:}
If we are in this case, the constraint on generator 4 is inactive, some of its capacities will remain unused. It means that generator 4 will lose money if it chooses to use (instead of keeping unused) these capacities.
\begin{lemma}
If generator 4 does not use all the capacity that it owns, it means that we have:
\begin{equation*}
    q_A \le e_A(y_4 + f_4)
\end{equation*}
\end{lemma}
\begin{proof}
    If generator 4 sells more electricity (a quantity $\delta_K$), it will earn:
    \begin{equation*}
        \delta \Pi = +\delta_K q_A - \delta_K  e_A (y_4 + f_4)
    \end{equation*}
    The additional profit made by selling the quantity $\delta_K$ minus the loss of profit due to the price reduction in the electricity market. This quantity is negative as it does not want to sell this capacity.
\end{proof}
\begin{corollary}
    Generator 4 would like to be paid by generator 3 more than:
    \begin{equation*}
        p_{34}^{PTR} \ge  e_A (y_4 + f_4) \ge q_A
    \end{equation*}
\end{corollary}
We can make the same reasoning regarding generator 3 and the maximum price that it is willing to pay for a quantity $\delta_K$:
\begin{lemma}
    Generator 3 is willing to buy a capacity $\delta_K$ to generator 4 if:
    \begin{equation*}
        p_{34}^{PTR} \delta K\le \delta \Pi = q_a\delta K - e_A (y_3+ f_3)\delta K
    \end{equation*}
\end{lemma}

\textbf{Conclusion:} Trades will be possible in this case as we would have if and only if:
\begin{equation}
\label{caseA.1}
 e_A (y_4+ f_4) \le p_{34}^{PTR} \le q_a - e_A (y_3+ f_3)
\end{equation}
Generator 4 will trade with generator 3 until its constraint becomes active (reach case 2) or that constraint on generator 3 becomes inactive or such that Equation \ref{caseA.1} becomes impossible (i.e. $q_A= e_A x_{BA})$. If both constraints are inactive, no one will be willing to buy or sell capacities as these capacities will remain unused and as a consequence have a value equal to zero.
The case where one generator keeps an unused capacity instead of selling it is called "right-witholdings"

\textbf{Case 2:}
\inComplete{
\newline We recall that we are in the case where constraints on 3 and 4 are active.
As constraints are active for both generators, any capacity that is traded between generators will be used to export from market B to market A. As a consequence, these transactions will not change the price of electricity in market A.
\newline So if generator 3 buys the capacity $\delta_K$ to generator 4, we have: 

$\forall s \in S, (\dfrac{\partial \Pi^{Rs}_3(K_k)}{\partial K_3} - \dfrac{\partial \Pi^{Rs}_3(K_k)}{\partial K_4})\delta_K= q_A^s \delta_K$.
\newline On the other hand, if generator 4 sells the capacity $\delta_K$ to generator 3, we have: 

$\forall s \in S, (\dfrac{\partial \Pi_4(K_k)}{\partial K_3} -\dfrac{\partial \Pi_4(K_k)}{\partial K_4})\delta_K= - q_A^s \delta_K$

So when we compute the value of the partial derivatives we have: 
\begin{equation*}
\label{resolsecondary}
      \dfrac{\partial \Pi_3^{Rs}(K_k)}{\partial K_3} - \dfrac{\partial \Pi_3^{Rs}(K_k)}{\partial K_4} + \dfrac{\partial \Pi_4^{Rs}(K_k)}{\partial K_3} -\dfrac{\partial \Pi_4^{Rs}(K_k)}{\partial K_4} = 0
\end{equation*}
}
\newline Generators are indifferent to trade between each other as trading between each other will not change prices in markets A and B it will just change the quantity sold. The marginal price of the PTRs market between 3 and 4 will be exactly $q_A^s$. The observations when generators from different countries trade between each other will be similar to that.
\paragraph*{Conclusion of the secondary market}

It is important to notice that when we will run the secondary market, the values of $\delta K_i$ will evolve and as a consequence, the dispatch of capacities and state of the constraints will also evolve.
To conclude, these different cases show that after running the secondary market we could be in a situation where some generators keep a remaining $\delta K_i$ unused even if some other generators have their constraint active and would like to buy an adding transmission capacity. This phenomenon is called ``right withholding''  and it can happen when there is generators who would like to export more while others are keeping some of their transmission capacities unused.
Regarding prices in the secondary market, Equation \eqref{pricebounds} provides a bound for the prices, we do not study the exact values of prices, we can get them by looking on the behaviour of sellers between each others.
  The resolution of the primary market exhibits the same phenomenon of rights witholdings.

\subsection{Possible regulatory policies to reduce market power and rights withholding}

\inComplete{  We can try to fine generators who do rights withholding, knowing the values of the equilibrium which drive prices of Physical transmission Rights would be helpful in doing so.}

Some markets can implement a policy that forces electricity producers to either sell or use their capacity (e.g. Use it or sell it, UIOSI, in Europe) by implementing a penalty if a generator does not sell its remaining capacity -- or simply removing its rights outright.
Alternatively, financial assistance could be provided with the loss of profit due to the sale of an extra capacity to the other generators, which could be financed with the income generated by the sales of capacities in the primary market. 
In either case, these could be seen as market power mitigation (MPM) policies \cite{10258359}.
In the next subsections, we discuss how the value of the marginal congestion cost $\eta$ can be set, 
and analyze the Use-it or Sell-it policy, and we point to its main weaknesses.

\subsubsection{Management of congestion costs by the choice of the value of congestion costs $\eta$}
 
 We saw that in case ``right withholding'', we will stay in a state where generators from 1/2 can keep an unused capacity at the end of the secondary market if generator 1 holds unused capacities and generator 3 wants to buy more transmission capacities:
\begin{align*}
    &3 D_A^s + e_A[36 f_1 + 36 f_2 + 2 f_4 + 14f_3] <  36 \alpha_A - 39(\alpha_B + \eta)\\&
    \delta K_1 = K_1^p - K_1^{max} > K_3^{max} - K_3^p
\end{align*}
\label{eta policy a eviter}
As a consequence, one can choose the value of $\eta$ to reduce cases where this situation happens. 
Reducing $\eta$ in a scenario $s\in S$ will also change the optimal values of production and increase the values of the maximum amount of electricity the generators want to export ($K_1^{max}$ for generator 1 and $K_3^{max}$ for generator 3). 
It will reduce the cases where \ref{eta policy a eviter} is satisfied, but it will not always increase the capacity available to the congested transmission line.
This situation (\ref{eta policy a eviter}) occurs when there is a low demand in the local market A. Generators 1 and 2 do not wish another competitor to enter their markets, and generators 3/4 are not willing to pay for the transmission capacities, as the price would always be very low in this market.
\newline It will be also dependent on the realisation of the local market B, since the maximum quantity that generators 1/2 are willing to export will depend on the parameters of this market.
\inComplete{
As generators will have to pay in order to buy capacities in the primary market, the primary market's revenue can be used to support the implementation of a negative $\eta$ if necessary.} 
It is an alternative to congestion rent and FTRs suggested by Joskow and Tirole \cite{Tirole2000}. This lends credit to the idea of a negative congestion cost $\eta$ as it will be a way to redistribute the profit made in the primary market.
\newline 
It will also give an interesting situation for generators as they could maybe buy some options to prevent them from these variations of $\eta$.
These options would create a market for Financial Transmission Rights (FTRs). However, Joskow and Tirole \cite{Tirole2000} show that these products tend to reinforce market power, which is undesirable.

\subsubsection{Model Use-it or Sell-it (UIOSI) or Use-it or loose it (UIOLI)}
UIOLI is the current model used in Europe regarding the allocation of cross-border capacities. (Notice there is no secondary market. Use-it or sell-it assumes there is a secondary market.)
 To assess the efficiency of such a policy (UIOSI), we will compare the production and price bounds of secondary market in PTRs.
We will use the following notation:
\begin{itemize}
    \item ${\Pi_1^B}_{keep}$ represents the profit if generator 1 decides to not use its remaining capacity
    \item ${\Pi_1^B}_{use}$ represents the profit if generator 1 decides to use its remaining capacity
    \item [$\bullet$] $\delta \Pi_1$ the marginal difference of profit for generator 1 if it uses its remaining capacity $\delta K$ (be careful, $\delta \Pi_1 \ne \dfrac{\partial \Pi_1(K_k)}{\partial K_1} = 0$ , because normally this capacity remains unused as it will decreased generator 1's profit). Given the definitions, we have: ${\Pi_1^B}_{use} = {\Pi_1^B}_{keep} + \delta \Pi_1\delta K$.
\end{itemize}
With that policy, we can now see that the equation \ref{pricebounds} that used to set the bounds for prices in this market is no longer true. Now, as generators 1/2 have only two options, sell or use, if they use the remaining capacity $\delta K$ in the case where they would prefer to keep it unused, it means that it would decrease their profit. Using this capacity $\delta K$ will change their profit:
\begin{align*}
    &{\Pi_1^B}_{use} = {\Pi_1^B}_{keep} + \delta \Pi_1 \times \delta K \\ &= {\Pi_1^B}_{1 keep} + \delta K (D_B - x_{total}^B - \delta K  - [x_{1}^B + \alpha_A + \eta]) 
\end{align*}
As generator A would prefer to keep this capacity unused, we know that:
${\Pi_1^B}_{use} - {\Pi_1^B}_{keep} = \delta \Pi_1 \times \delta K \le 0$.
\newline Generator 1 would prefer to sell to generator 3 (it will be the same condition with 4) instead of using the remaining capacity $\delta K$ if:
$$\begin{aligned}
    p^{PTR}_{13} \delta K + \dfrac{\partial \Pi_1(K_k)}{\partial K_3} \delta K -\delta \Pi_1 \times \delta K > 0 \\
    \Leftrightarrow p^{PTR}_{13} \ge \delta \Pi_1 - \dfrac{\partial \Pi_1(K_k)}{\partial K_3}
\end{aligned}$$
This policy lowers the lower bound from equation \ref{minprice} as there are incentives to sell instead of keeping unused capacities $\delta \Pi_1 \le 0 = \dfrac{\partial \Pi_1(K_k)}{\partial K_1}$. The issue is that in some cases, some generators will overproduce because they do not want to sell their remaining capacity to a competitor.

In the European implementation, Use-it or Loose-it policies may incentivize generators to overproduce and not encourage producers from other markets to export even if a small amount of additional capacity will be available to them with these policies.
On the other hand, the Clean Energy Package suggested 70\% minimum utilization of the interconnectors' capacity in Europe \cite[Art. 16.8]{IME}, 
effectively tasking the regulators to address any such issues. 

\section{Conclusion and future work}

This study demonstrated the crucial role of demand estimates and pricing cross-border transmission in regulating electricity markets. 
By lowering prices in the markets preceding the spot market, one can effectively promote social welfare under certain market conditions. 
Despite the lower prices, competitors have an incentive to participate in the day-ahead and intraday markets to avoid the potential disadvantages that could result from a prisoner's dilemma \cite{tucker1959contributions}. 
When cooperation between producers is not possible, this pricing approach may still be advantageous for social welfare, 
as the only Nash equilibrium will be for all producers to enter the day-ahead market even if prices are lower. 
However, they will slightly reduce their production in this case, compared with the case where the price is the expected price in the spot market.


In addition, we examined the phenomenon of rights withholding in the transmission-capacity market and found that the current policy (UIOSI or UIOLI) of encouraging full capacity utilization or sales can lead to a production surplus.
To solve this problem, one could encourage exports with a negative congestion cost, which would encourage greater utilization of transmission capacity.

\inComplete{
In this study, we acknowledge that we have not taken into account the variations between generators, in particular the different activation times and cost structures associated with each type of resource. To improve our work, we propose a market design that takes these differences into account, particularly in a multi-stage electricity market where different types of generators, such as nuclear power plants, wind turbines and hydro turbines, coexist and will be used depending on the time before delivery contracts are executed. For example, energy generated by a nuclear power plant is more likely to be traded on a day-ahead market, as the activation time is very high, whereas renewable energies such as wind or solar power are more likely to be traded on the spot market, as they depend on a good production forecast. By taking into account different activation times and cost structures, market design can be adapted to encourage specific types of electricity generation.
}

\bibliographystyle{ieeetr}
\bibliography{refs}

\section*{Acknowledgements}
 This work received funding from the National Centre for Energy II (TN02000025).
 We would also like to acknowledge many fruitful discussions with Claudia Alejandra Sagastizabal.

\clearpage
\onecolumn

\appendix

\subsection{Generation game}
In the first model (see \ref{B-Generation game}), we have to solve this system of equations to get the values of spot sales $x_1$ and $x_2$ depending on the forward sales $f_1$ and $f_2$.
\label{app: 1-B}
$$
\left\{
    \begin{array}{ll}
        \dfrac{\partial U^R_1(x_1;f_1,f_2, x_2)}{\partial x_1}=0 \\
        \dfrac{\partial U^R_2(x_2;f_1,f_2, x_1)}{\partial x_2}=0 
    \end{array}
\right.
\Leftrightarrow
\left\{
    \begin{array}{ll}
      D - e(2x_1 + x_2 - f_1) - \alpha_1 = 0\\
     D - e(2x_2 + x_1 - f_2) - \alpha_2 = 0
    \end{array}
\right.
\Leftrightarrow
    \left\{
    \begin{array}{ll}
        x_1 = \frac{1}{2}\left(\frac{D- \alpha_1}{e}- x_2 + f_1\right)\\
        x_2 = \frac{1}{2}\left(\frac{D- \alpha_2}{e}- x_1 + f_2\right)
    \end{array}
\right.
$$
\newline Then, we compute the Nash equilibrium of this game (i.e. solve the system of equations given by these two reponse functions), it gives us:
\begin{equation*}
   x_1(f_1,f_2) =\frac{\frac{1}{e}(D - 2\alpha_1 + \alpha_2) + 2f_1 - f_2}{3} \\
\end{equation*}
\begin{equation*}
x_2(f_1,f_2) =\frac{\frac{1}{e}(D - 2\alpha_2 + \alpha_1) + 2f_2 - f_1}{3} \\
\end{equation*}
\begin{equation*}
q(f_1,f_2) =\frac{\frac{1}{e}(D + \alpha_1 + \alpha_2) - f_1 - f_2}{3} 
\end{equation*}

\subsection{Generation game with constraints}
\label{app: 2-B}
When we compute the Lagragian of the problem (see \ref{lagrange}), the KKT conditions give us the following system of equations:
\newline We will use this notation: $\forall s \in S, \forall j \in I,
$\newline$
\begin{aligned}[t] q^{j,s}_A(\{y^s_i,f_i\}_{[\![1;4]\!]}) &= D_A^s - e_A \sum_{i \ne j}[y_i^s + f_i ] - e_A f_j  = q_A^s(\{y_i^s,f_i\}_{[\![1;4]\!]}) + e_A y_j^s\end{aligned}$.
\newline As a consequence we can rewrite the values of the Lagrangian of the problem is, in the spot market A (with $\forall i \in \{3,4\},\lambda_j^{0s} \ge 0$ the Lagrange multipliers associated with the constraint $F_j^{Rs}$) :
\newline
$\forall s \in S,\forall j \in \{1,2\},$
\newline
$$\begin{aligned}
    L_j^{Rs}(\{f_i,y_i^s\}_{i \in I},\eta) &= -U^{Rs}_j(\{f_i, y_i^s\}_{i \in I}, \eta)
          = -y_j^s (q^{j,s}_A(\{y_i^s,f_i\}_{i \in I}) - e_Ay_j^s)+ \alpha_A (y_j^s + f_j) 
\end{aligned}$$
 \newline
$\forall s \in S,\forall j \in \{3,4\},$
\newline
$$\begin{aligned}
    L_j^{Rs}(\{f_i, y_i^s\}_{i \in I},\eta,\lambda^{0s}_j) &= -U^{Rs}_j(\{f_i, y_i\}_{i \in I}, \eta) + \lambda^{0s}_j F^{Rs}_j(\{f_i, y_i^s\}_{i \in I})\\
        & = -y_j^s (q^{j,s}_A(\{y_i^s,f_i\}_{i \in I}) - e_Ay_j^s)+ (\alpha_B + \eta + \lambda^{0s}_j) (y_j^s + f_j) - \lambda_j^{0s}K_j
\end{aligned}$$
\newline We recall that $q_A^s$ represent the price in the spot market depending on the production of each generator. This notation will be useful as $\forall j \in I, q^{j,s}_A$ does not depend on $j$ and so: \label{derive}
\newline $\dfrac{\partial q^{j,s}_A(\{y^s_i,f_i\}_{[\![1;4]\!]})}{\partial y_j^s} = \dfrac{\partial D_A^s - e_A (\sum_{i \ne j}[y_i^s + f_i ] - f_j )}{\partial y_j^s} = 0 $. 

$$   
         \forall j \in I , \dfrac{\partial L^{Rs}_j(\{f_i,y_i^s\}_{i \in I},\eta,\lambda^{0s}_j)}{\partial y_j^s} = 0   
\Leftrightarrow
\left\{
    \begin{array}{ll}
        \forall j \in\{1,2\}, y^s_j = \frac{q_A^{j,s}(\{f_i,y_i^s\}_{i \in I}) - \alpha_A}{2 e_A}\\
        \forall j \in\{3,4\}, y^s_j = \frac{q_A^{j,s}(\{f_i,y_i^s\}_{i \in I}) - \alpha_B -\eta - \lambda^{0s}_j}{2 e_A}
    \end{array}
\right.\\
$$
\newline
$$
\Leftrightarrow
\left\{
    \begin{array}{ll}
        \forall j \in\{1,2\}, y^s_j = \frac{D_A^s - e_A \sum_{i\ne j}(y_i^s + f_i)- e_A f_j  - \alpha_A}{2 e_A}\\
        \forall j \in\{3,4\}, y^s_j = \frac{D_A^s - e_A \sum_{i\ne j}(y_i^s + f_i) - e_A f_j- \alpha_B -\eta - \lambda^{0s}_j}{2 e_A}
    \end{array}
\right.
\Leftrightarrow
\left\{
    \begin{array}{ll}
        \forall j \in\{1,2\}, y^s_j = \frac{D_A^s - e_A \sum_{i\ne j}x_j^{As}- e_A f_j  - \alpha_A}{2 e_A}\\
        \forall j \in\{3,4\}, y^s_j = \frac{D_A^s - e_A \sum_{i\ne j}x_j^{As} - e_A f_j- \alpha_B -\eta - \lambda^{0s}_j}{2 e_A}
    \end{array}
\right.\\
$$
\newline
$$
\Leftrightarrow
\left\{
    \begin{array}{ll}
        \forall j \in\{1,2\}, x_j^{As} + \frac{1}{2}\sum_{i \ne j} x_i^{As} = \frac{D_A^s - \alpha_A + e_A f_j}{2 e_A}\\
        \forall j \in\{3,4\}, x_j^{As} + \frac{1}{2}\sum_{i \ne j} x_i^{As} = \frac{D_A^s - \alpha_B - \eta - \lambda^{0s}_j + e_A f_j }{2 e_A}
    \end{array}
\right.\\
$$
\subsection{Day-ahead market}
\label{app: 2-B day ahead market}
Given the KKT conditions, we know that $\forall j \in I,$ a solution ($f_j,\lambda_j^{1}$) is a saddle point of the Lagrangian $L_j^D(f_j,\lambda_j^{1})$.
\newline So, we minimize these equations in $\{f_j\}_{j \in I}$, we solve these equations:
$\forall j \in I ,\dfrac{\partial L^D_j(\{f_i\}_{i \in I},\eta,\lambda^{1}_j)}{\partial f_j} = 0 $.
\bigskip\newline This gives us the reponse functions $\forall j \in I, f_j(f_{i \ne j})$. As we will minimize over the expected profit, some average values will appear:
\begin{itemize}
    \item $C_A= \sum_{s\in S} p_s C_A^s$
    \item $\forall j \in I, \lambda_j^{0}= \sum_{s\in S} p_s \lambda_j^{0s}$
    \item $q_A=\sum_{s\in S} p_s q_A^s$
\end{itemize}

 $\forall j \in I, \dfrac{\partial L^D_j(\{f_i\}_{i \in [\![1;4]\!]},\eta,\lambda^{1}_j)}{\partial f_j} = 0 $
\newline$$\Leftrightarrow 
\left\{
    \begin{array}{ll}
        \forall j \in\{1,2\}, \sum_{s\in S} p_s[\beta_A^s + \frac{4}{5}(C_A^s - \frac{e_A}{5}\sum_{k \in I}f_k - \alpha_A) - \frac{e_A}{5}(r_j^s - \frac{1}{5}\sum_{k \ne j}f_k +\frac{4}{5}f_j)]=0\\
        \forall j \in\{3,4\}, \sum_{s\in S} p_s[\beta_A^s + \frac{4}{5}(C_A^s - \frac{e_A}{5}\sum_{k \in I}f_k - \alpha_B -\eta) - \frac{e_A}{5}(r_j^s - \frac{1}{5}\sum_{k \ne j}f_k +\frac{4}{5}f_j) -\lambda^1_j]=0
    \end{array}
\right.$$
$\Leftrightarrow 
\left\{
    \begin{array}{ll}
        \forall j \in\{1,2\}, \beta_A + \frac{4}{5}(C_A - \frac{e_A}{5}\sum_{k \in I}f_k - \alpha_A) - \frac{e_A}{5}(r_j - \frac{1}{5}\sum_{k \ne j}f_k +\frac{4}{5}f_j)=0\\
        \forall j \in\{3,4\}, \beta_A + \frac{4}{5}(C_A - \frac{e_A}{5}\sum_{k \in I}f_k - \alpha_B -\eta) - \frac{e_A}{5}(r_j - \frac{1}{5}\sum_{k \ne j}f_k +\frac{4}{5}f_j) -\lambda^1_j=0
    \end{array}
\right.$
\newline
\newline$\Leftrightarrow 
\left\{
    \begin{array}{ll}
        \forall j \in\{1,2\}, \beta_A + \frac{4}{5}(C_A- \alpha_A)- \frac{e_A}{5}r_j-\frac{3e_A}{25}\sum_{k \ne j}f_k - \frac{8 e_A}{25}f_j=0\\
        \forall j \in\{3,4\}, \beta_A + \frac{4}{5}(C_A- \alpha_B - \eta )- \frac{e_A}{5}r_j-\frac{3e_A}{25}\sum_{k \ne j}f_k - \frac{8 e_A}{25}f_j- \lambda_j^1=0
    \end{array}
\right.$
\newline$\Leftrightarrow 
\left\{
    \begin{array}{ll}
        \forall j \in\{1,2\}, \frac{8}{5} f_j + \frac{3}{5}\sum_{i \ne j} f_i = \frac{3}{5 e_A}(D_A -3 \alpha_A+ 2 (\alpha_B + \eta) + \lambda_0^3 + \lambda_0^4) + \frac{1}{e_A}\beta_A\\
        \forall j \in\{3,4\}, \frac{8}{5} f_j + \frac{3}{5}\sum_{i \ne j} f_i = \frac{3}{5 e_A}( D_A -3 (\alpha_B + \eta)+ 2 \alpha_A  -4 \lambda_0^{j} + \lambda_0^{-j}) + \frac{1}{e_A}(\beta_A - \lambda^1_j)
    \end{array}
\right.$
\newline Then, we compute the Nash equilibrium for these solutions, by finding the solution of this matrix system $AF_A=B$ with:
\bigskip \newline $A = 
  \left[ {\begin{array}{cccc}
    \frac{8}{5} & \frac{3}{5} & \frac{3}{5} & \frac{3}{5}\\
    \frac{3}{5} & \frac{8}{5} & \frac{3}{5} & \frac{3}{5} \\
    \frac{3}{5} & \frac{3}{5} & \frac{8}{5} & \frac{3}{5} \\
    \frac{3}{5} & \frac{3}{5} & \frac{3}{5}&\frac{8}{5}\\
         \end{array} } \right]\; 
  F_A =
  \left[ {\begin{array}{c}
    f_1 \\
    f_2 \\
    f_3\\
    f_4 \\
  \end{array} } \right]\; 
  B = 
  \frac{1}{e_A}\left[ {\begin{array}{c}
    \frac{3}{5 }(D_A -3 \alpha_A+ 2 (\alpha_B + \eta) + \lambda_0^3 + \lambda_0^4) +\beta_A \\
    \frac{3}{5 }(D_A -3 \alpha_A+ 2 (\alpha_B + \eta) + \lambda_0^3 + \lambda_0^4) + \beta_A \\
    \frac{3}{5 }( D_A -3 (\alpha_B + \eta)+ 2 \alpha_A  -4 \lambda_0^{3} + \lambda_0^{4}) + \beta_A- \lambda^1_3\\
    \frac{3}{5}( D_A -3 (\alpha_B + \eta)+ 2 \alpha_A  -4 \lambda_0^{4} + \lambda_0^{3}) + \beta_A- \lambda^1_4\\
  \end{array} } \right]$
\bigskip \newline We have the following solutions for day-ahead sales $f_i$:
\newline
$\begin{aligned}\label{fsanscontraintes}
    \forall i \in \{1,2\}, f_i = \frac{1}{17 e_A}[3(D_A  -9 \alpha_A + 8 (\alpha_B+\eta)  + 4 (  \lambda^0_3 +4\lambda^0_4))+ 3(\lambda_3^1 + \lambda_4^1)+ 5\beta_A]
\end{aligned}$
\newline $
\begin{aligned}
     \forall i \in \{3,4\}, f_i = \frac{1}{17 e_A}[3(D_A -9 (\alpha_B+\eta) + 8\alpha_A  - 13 \lambda_{j}^0+ 4 \lambda_{-j}^0) -14\lambda_{j}^1 + 3\lambda_{-j}^1+ 5\beta_A]\\
\end{aligned}$
\subsection{Optimal arbitrageurs profit}
\label{app: 2-B arbitrageurs profit}
When optimizing the optimal arbitrageurs (see \ref{arbitrage}) profit, we compute the partial derivatives regarding $\beta_A$, we have:
\begin{gather*}
\dfrac{\partial z(\beta_A)}{\partial \beta_A} =
    q_A - \frac{10}{17e_A}(\alpha_A + \alpha_B + \eta) - \sum_{i=1}^4 f_i - \frac{20}{17e_A}\beta_A=0\\
    \Leftrightarrow \frac{1}{17 }(D_A + 8(\alpha_A + \alpha_B + \eta)+4(\lambda_3^0 + \lambda_4^0) + \frac{3}{5} (\lambda_3^1 + \lambda_4^1)- 4 \beta_A) - \frac{10}{17e_A}(\alpha_A + \alpha_B + \eta) \\- \frac{1}{17 e_A}(12 D_A -3 (\lambda_3^0 + \lambda_4^0) - 6 (\alpha_A + \alpha_B + \eta)- 3 (\lambda_3^1 + \lambda_4^1) + 20 \beta_A)
     - \frac{20}{17e_A}\beta_A=0\\
     \Leftrightarrow e_A(D_A + 8(\alpha_A + \alpha_B + \eta)+4(\lambda_3^0 + \lambda_4^0) + \frac{3}{5} (\lambda_3^1 + \lambda_4^1)- 4 \beta_A) - 10(\alpha_A + \alpha_B + \eta) \\-(12 D_A -3 (\lambda_3^0 + \lambda_4^0) - 6 (\alpha_A + \alpha_B + \eta)- 3 (\lambda_3^1 + \lambda_4^1) + 20 \beta_A)
     - 20\beta_A=0\\
     \Leftrightarrow\beta_A = \frac{D_A(e_A -12) + (\alpha_A + \alpha_B + \eta)(8e_a - 4) + (\lambda_3^0 + \lambda_4^0)(4e_A +3)+ (\lambda_3^1 + \lambda_4^1)(\frac{3}{5}e_A+3) }{4 e_A + 40}\\
\end{gather*}
\begin{equation*}
     \Leftrightarrow D_A^{SO} = \frac{D_A(5e_A + 28) + (\alpha_A + \alpha_B + \eta)(8e_a - 4) + (\lambda_3^0 + \lambda_4^0)(4e_A +3)+ (\lambda_3^1 + \lambda_4^1)(\frac{3}{5}e_A+3) }{4 e_A + 40}
\end{equation*}

\end{document}